\documentclass[11pt]{amsart}
\usepackage{amsmath,amstext,amsbsy,amssymb,amscd}
\usepackage{amsmath}
\usepackage{amsxtra}
\usepackage{amscd}
\usepackage{amsthm}
\usepackage{amsfonts}
\usepackage{graphicx}%
\usepackage{amssymb}
\usepackage{eucal}
\usepackage{color}
\usepackage{multirow}
\usepackage[enableskew,vcentermath]{youngtab}

\newtheorem{thm}{Theorem}[section]
\theoremstyle{plain}

\newtheorem{lem}[thm]{Lemma}
\newtheorem{prop}[thm]{Proposition}
\newtheorem{cor}[thm]{Corollary}

\theoremstyle{definition}

\theoremstyle{remark}
\newtheorem{rem}[thm]{Remark}

\newtheorem{question}[thm]{Question}

\definecolor{A}{rgb}{.75,1,.75}

\numberwithin{equation}{section}

\newcommand{\ds}{\displaystyle}

\newcommand{\C}{\mathbb C}
\newcommand{\Z}{\mathbb Z}

\newcommand{\var}{\varepsilon}

\newcommand{\F}{\mathbb F}
\newcommand{\la}{\lambda}
\newcommand{\ga}{\gamma}
\newcommand{\Ga}{\Gamma}
\newcommand{\De}{\Delta}
\newcommand{\na}{\nabla}
\newcommand{\om}{\omega}
\newcommand{\St}{{\rm St}}
\newcommand{\sym}{S^{\bullet}(V)}
\newcommand{\wedg}{\wedge^{\bullet}(V)}
\newcommand{\Det}{{\rm Det}}
\newcommand{\Gq}{GL_n(q)}
\newcommand{\GF}{GL_n(\F)}
\newcommand{\Grk}{\mathcal{G}_r(\F)}

{\vskip-\lastskip\medskip
  \noindent
  {\em #1.}\enspace
  }%
{\qed\par\medskip
  }

\begin{document}

\title[The $GL_n(q)$-module structure
of the symmetric algebra]{The $GL_n(q)$-module structure
of the symmetric algebra around the Steinberg module}
\author[Wan and Wang]{Jinkui Wan and Weiqiang Wang}

\address{
Department of Mathematics,
Beijing Institute of Technology,
Beijing, 100081, P.R. China. }
\email{wjk302@gmail.com}

\address{Department of Mathematics, University of Virginia,
Charlottesville,VA 22904, USA.}
\email{ww9c@virginia.edu}

\begin{abstract}
We determine the graded composition multiplicity in the symmetric
algebra $\sym$ of the natural $\Gq$-module $V$, or equivalently in
the coinvariant algebra of $V$, for a large class of irreducible
modules around the Steinberg module. This was built on a
computation, via connections to algebraic groups, of the Steinberg
module multiplicity in a tensor product of $\sym$ with other tensor
spaces of fundamental weight modules.
\end{abstract}

\maketitle

\section{Introduction}

The symmetric algebra $\sym$ is naturally a graded module over the
finite general linear group $GL_n(q)$, where $V =\F^n$ is the
standard $GL_n(q)$-module over an algebraically closed field $\F$ of
characteristic~$p$ and $q=p^r$ for $r\geq 1$. Dickson's classical
theorem \cite{Di} states that the algebra of $GL_n(q)$-invariants in
$\sym$ is a polynomial algebra in $n$ generators, and this work
served as the starting point of all the subsequent works on the
$\Gq$-module structure of $\sym$ and closely related modules.

The composition multiplicity of the Steinberg module $\St$ in $\sym
\otimes \wedg \otimes \Det^k$, where $\wedg$ denotes the exterior
algebra of $V$ and $\Det$ denotes the determinant module, has been
determined in various special cases by Kuhn, Mitchell, and Priddy
\cite{KM, Mi, MP} and in full generality by the authors \cite{WW}.
The topological approach of \cite{Mi, MP} using Steenrod algebra
worked only in a prime field and it is not clear how to develop
further along this line. On the other hand, the approach of
\cite{KM, WW} is based on a modular version of a formula of Curtis
in terms of parabolic subgroup invariants (for closely related work
see \cite{Mui, MT}). These parabolic subgroup invariants were
determined in a constructive manner, and it seems difficult to
extend the approach much further.

By a basic observation of Mitchell \cite{Mi}, finding the graded
multiplicity of a simple module $L$ in the symmetric algebra $\sym$
is equivalent to finding the graded composition multiplicity of $L$
in the coinvariant algebra of $V$ which is a graded regular
representation of $\Gq$. A full answer for every simple $\Gq$-module
is beyond the reach for now as it would imply the degrees of all
principal indecomposable modules (PIMs).

The main goal of the paper is to find an elegant closed formula for
the graded composition multiplicity in $\sym$ for a large class of
simple $\Gq$-modules around the Steinberg module (i.e., simple
modules of highest weights not far from the Steinberg weight
$(q-1)\rho$). The twisting by the determinant module plays an
important role in this paper.

Our new approach is based on the intimate and deep connections
between representations of $\Gq$ and of the algebraic group $\GF$,
and it is a two-step process. First,  via connections to algebraic
groups, we compute the graded composition multiplicity of  $\St$ in
various tensor modules of the form $\sym \otimes N$ for some natural
$\GF$-modules $N$. Secondly, such a composition multiplicity of
$\St$ when combined with classical results on PIMs of $\Gq$ are used
to derive a closed formula for the graded composition multiplicity in
$\sym$ for a large class of simple $\Gq$-modules around the
Steinberg module (up to twists by $\Det$).

Let us explain in some detail. Bendel, Nakano and Pillen \cite{BNP}
has recently developed an amazing link between the
$\text{Ext}$-groups of $\Gq$ and of $\GF$, and used it to find upper
bounds for cohomology of finite groups. As explained to us by Pillen
(see Section~\ref{sec:Preliminary}), the machinery of \cite{BNP} can
be used effectively to transform the problem of computing the
Steinberg module multiplicity in a rational $\GF$-module with a good
filtration (viewed as a $\Gq$-module) into a problem of counting
multiplicities in infinitely many rational $\GF$-modules with good
filtrations. By a classical result of J.-P. Wang \cite{Wa}, the
latter becomes essentially a highly nontrivial combinatorial problem
of counting multiplicities of irreducible characters in
characteristic zero. In our cases of interest, the intricate
combinatorial problem can be eventually solved with a key tool being
the Pieri formula. In this way, we are able to determine the graded
multiplicity of the Steinberg module in $\sym\otimes
\wedge^m(V)\otimes \Det^k$ (see Theorem~\ref{thm:symwedg}) and more
generally in $\sym \otimes \wedge^\nu(V) \otimes \Det^k$ for
suitable partitions $\nu$ and suitable $k$ (see
Theorem~\ref{symwedgnu}). Theorem~\ref{thm:symwedg} recovers in a
different form one of the main results in \cite[Theorem~C]{WW}.

Note that $\text{Hom}_{\Gq} (\St, \sym \otimes N) \cong
\text{Hom}_{\Gq} (\St \otimes N^*, \sym)$ for a finite dimensional
$\Gq$-module $N$, and that $\St \otimes N^*$ is projective. The
results of Ballard on PIMs \cite{Ba} (which was inspired by
Humphreys and Verma \cite{HV} and improved by Chastkofsky \cite{Ch}
and Jantzen \cite{J1}) allow us to find an explicit decomposition of
$\St \otimes N^*$ for suitable $N$ into a direct sum of PIMs. We
derive from this and Theorem~\ref{symwedgnu} a closed formula for the
graded composition multiplicity in $\sym$ for a large class of
simple modules around the Steinberg module; see
Theorem~\ref{thm:grcompmult2}. In light of an observation in
\cite{Mi}, Theorem~\ref{thm:grcompmult2} affords an equivalent
reformulation in terms of the coinvariant algebra of $V$ in place of
$\sym$; see Theorem~\ref{th:coinv}. Also, from
Theorem~\ref{thm:symwedg} and results of Tsushima \cite{Ts}  on PIMs
(a special case of which goes back to Lusztig \cite{Lu}), we recover
the main results of Carlisle and Walker \cite{CW}, who obtained a
multiplicity formula for several simple modules very close to $\St$
in $\sym$ using an ingenious combinatorial and semigroup approach.

Our work opens a new and effective way of studying the $\Gq$-module
structure of $\sym$ via its connection to algebraic groups. At the
end of the paper, we formulate several open problems, and speculate
a formula on the composition multiplicity in the socle of $\sym$ for
a family of simple modules.

The paper is organized as follows. In Section~\ref{sec:Preliminary}
we recall the basics of the algebraic group $\GF$ and of the finite
group $\Gq$ (a basic reference in this direction is the book of
Humphreys \cite{Hu}), and formulate a key formula derived from
\cite{BNP}. We determine the graded multiplicity of $\St$ in the
tensor products of $\sym$ with various natural $\Gq$-modules in
Section~\ref{sec:Stmult}. This is then applied in
Section~\ref{sec:compmult} to determine an explicit formula for the
graded composition multiplicity in $\sym$, or equivalently in the
coinvariant algebra of $V$, for a large class of simple
$\Gq$-modules.

\vspace{.2cm}

{\bf Acknowledgments.} We are indebted to Cornelius Pillen for
explaining to us a  key consequence of results in \cite{BNP} which
has played a fundamental role in our work. In addition, we thank Jim
Humphreys, Nick Kuhn, and Leonard Scott for stimulating discussions
and helpful references. The research of the second author is
partially supported by NSF grant DMS-0800280.


\section{The preliminaries}\label{sec:Preliminary}


\subsection{Finite group $\Gq$ and algebraic group $\GF$}

Let $GL_n(\F)$ be the general linear group over an algebraically
closed field $\F$ of prime characteristic $p>0$. Let $T$ be the
maximal torus consisting of diagonal matrices in $\GF$ and $B$ be
the Borel subgroup consisting of upper triangular matrices. Denote
by $\Phi^+$ (resp. $\Phi^-$) the corresponding positive (resp.
negative) root system. Then we have the Weyl group $W = S_n$, the
set of simple roots $\Pi=\{\alpha_1,\ldots,\alpha_{n-1}\}$ and the
normalized bilinear form satisfying  $(\alpha,\alpha)=2$ for
$\alpha\in\Phi$.
The simple coroot $\alpha_i^{\vee}=\frac{2\alpha_i}{(\alpha_i,\alpha_i)}$ coincides with $\alpha_i$
for $1\leq i\leq n-1$.
Let $X=X(T)$ be the integral weight lattice which
can be identified with $\Z^n$ and denote the set of dominant
integral weights by
$$
X^+ =\{\la~|~\la=(\la_1,\ldots,\la_n)\in\Z^n,
\la_1\geq\la_2\geq\cdots\geq\la_n \}.
$$

For $\la\in X^+$, there exists a simple $\GF$-module $L(\la)$ of
highest weight $\la$. These $\GF$-modules are pairwise
non-isomorphic and exhaust the isomorphism classes of simple
$\GF$-modules. For $\la\in X^+$, let $\na(\la):= {\rm
ind}^{\GF}_B\la$ be the induced module and $\Delta(\la) :=
\na(-w_0\la)^*$ be the Weyl module of highest weight $\la$, where
$w_0$ is the longest element in $W$. It is known that $\na(\la)$ has
a unique simple submodule isomorphic to $L(\la)$ and $\Delta(\la)$
has a unique simple quotient isomorphic to $L(\la)$.

Let $\text{Fr}: \GF\rightarrow \GF$ denote the Frobenius map, and
let $q=p^r$ for $r \ge 1$. The fixed point subgroup of the $r$th
iterate of the Frobenius map can be identified with $GL_n(q)$.
Denote the set of $q$-restricted weights in $X^+$ by
$$
X_r =\{\lambda\in X^+~|~0\leq\la_n<q, (\lambda,\alpha_i^{\vee})<q, 1\leq i\leq n-1\}. 
$$
The restrictions to $\Gq$ of the simple $\GF$-modules $L(\lambda)$
with $\lambda\in X_r$ form a complete set of pairwise non-isomorphic
simple $\Gq$-module (cf. \cite{Hu}, \cite[II.3]{J}). In particular,
the restriction of $L((q-1)\rho)$ to $\Gq$ is called the Steinberg
module and denoted by $\St =\St_r$, where
$
\rho=(n-1,n-2,\ldots,1,0).
$
We shall also write
\begin{equation}  \label{eq:rhoi}
\rho_i =n-i, \qquad  i=1, \ldots, n.
\end{equation}

Recall that a $\GF$-module $N$ has a good filtration (also called a
$\na$-filtration) if it admits a filtration with successive
quotients of the form $\na(\la)$, $\la\in X^+$ \cite[II 4.16]{J}.
Denote by $[N:\na(\la)]$ the multiplicity of $\na(\la)$ appearing in
a good filtration of $N$. We have the following lemma (cf. \cite[II,
Proposition~4.16]{J}).
\begin{lem}\label{lem:Janzten}
Let $N$ be a $\GF$-module admitting a good filtration. Then, for
each $\la\in X^+$,
\begin{align*}
[N:\na(\la)] &={\rm dim~Hom}_{\GF}(\De(\la),N),  \\
{\rm Ext}^i (\Delta(\la), N) & = 0, \quad \forall i\geq 1.
\end{align*}
\end{lem}
The following fundamental result is due to J.-P. Wang~\cite{Wa} (cf.
\cite[II, Proposition~4.19]{J}).
\begin{lem}  \cite{Wa}  \label{lem:J.Wang}
If $N$ and $N'$ are $\GF$-modules admitting good filtrations, then
the tensor product $N\otimes N'$ also has a good filtration.
\end{lem}


\subsection{Relating $\Gq$ to $\GF$}\label{subsec:lem}

Define the induced $\GF$-module
$$
\Grk={\rm ind}^{\GF}_{\Gq}(\F).
$$
The basic properties of $\Grk$ (for more general reductive groups)
were described by Bendel, Nakano and Pillen, and then used to give
upper bounds on dimensions of cohomology of finite groups of Lie
type.

\begin{lem}\cite[Proposition~2.3]{BNP}\label{lem:BNP}
Let $M,N$ be rational $\GF$-modules. Then,
\begin{align*}
{\rm Ext}^i_{\Gq}(M,N) \cong {\rm Ext}^i_{\GF}(M,N\otimes \Grk),
\quad \forall i \ge 0.
\end{align*}
\end{lem}
\begin{lem}\cite[Proposition~2.4]{BNP}\label{lem:BNP1}
As a $\GF$-module, $\Grk$ has a filtration with factors
$\na(\la)\otimes\na(-w_0\la)^{(r)}$ of multiplicity one for each
$\la\in X^+$.
\end{lem}

The proof of the following useful proposition, which follows from
the above results of \cite{BNP}, was communicated to us by Cornelius
Pillen.

\begin{prop}\label{Pillen}
Let $N$ be a finite dimensional rational $\GF$-module admitting a
good filtration. Then
\begin{align}
{\rm dim}~{\rm Hom}_{\Gq}(\St,N)
=\sum_{\la\in X^+}[N\otimes\na(\la):\na((q-1)\rho+q\la)].\label{eqn:Pillen}
\end{align}
\end{prop}

\begin{proof}
We first observe by Lemma~\ref{lem:Janzten} and \cite[II, 3.19]{J} that
\begin{align}
{\rm Ext}^i_{\GF} &(\St,N\otimes\na(\la)\otimes\na(-w_0\la)^{(r)})\notag\\
&\cong{\rm Ext}^i_{\GF}(\St\otimes\De(\la)^{(r)},N\otimes\na(\la))\notag\\
&\cong{\rm Ext}^i_{\GF}(\De((q-1)\rho+q\la),N\otimes\na(\la))\notag\\
&=
\left\{
\begin{array}{ll}
{\rm Hom}_{\GF}(\De((q-1)\rho+q\la),N\otimes\na(\la)),& \text{ if } i=0\\
0,& \text{ if } i>0.
\end{array}
\right.
\label{eqn:Pillen1}
\end{align}
It follows that ${\rm Hom}_{\GF}(\St,N\otimes\na(\la)\otimes\na(-w_0\la)^{(r)})$
is nonzero only  for finitely many $\la\in X^+$.

Let $\mathcal G$ be a $\Gq$-module which has a
$\na(\la)\otimes\na(-w_0\la)^{(r)}$-filtration.
For a $\GF$-submodule $S$ of $\mathcal G$ having a
$\na(\la)\otimes\na(-w_0\la)^{(r)}$-filtration,
we define a $\GF$-module $Q$ by the short exact sequence
$$
0\longrightarrow S \longrightarrow \mathcal G \longrightarrow Q \longrightarrow 0.
$$
Then $Q\cong\mathcal G/M$ also has a
$\na(\la)\otimes\na(-w_0\la)^{(r)}$-filtration.
The short exact sequence induces a long exact sequence with initial terms
\begin{align*}
0\longrightarrow {\rm Hom}_{\GF}(\St, N\otimes S)
& \longrightarrow {\rm Hom}_{\GF}(\St, N\otimes \mathcal G)  \\
&\longrightarrow {\rm Hom}_{\GF}(\St, N\otimes Q)
\longrightarrow {\rm Ext}^1_{\GF}(St, N\otimes S)
\end{align*}
where the last term ${\rm Ext}^1$ vanishes by~(\ref{eqn:Pillen1}).
Hence we obtain the following identity:
\begin{align*}
 \dim {\rm Hom}_{\GF}(\St, N\otimes \mathcal G)
 =\dim {\rm Hom}_{\GF}(\St, N\otimes S) +
 \dim {\rm Hom}_{\GF}(\St, N\otimes Q).
\end{align*}
By repeatedly applying this identity, we obtain that
\begin{align} \label{eq:infsum}
\dim {\rm Hom}_{\GF}(\St, N\otimes\Grk)
& =\sum_{\la\in X^+} \dim {\rm Hom}(\St, N\otimes\na(\la)\otimes\na(-w_0\la)^{(r)})
\end{align}
where all but finitely many summands on the right-hand side are zero.

By  \eqref{eqn:Pillen1}, \eqref{eq:infsum}, Lemma~\ref{lem:Janzten},
Lemma~\ref{lem:J.Wang}, and Lemma~\ref{lem:BNP} (for $M=\St$), we obtain that
\begin{align*}
\dim {\rm  Hom}_{\Gq}(\St, N)
=&\sum_{\la\in X^+}\dim {\rm Hom}
(\St, N\otimes\na(\la)\otimes\na(-w_0\la)^{(r)})  \\
=&\sum_{\la\in X^+}\dim {\rm Hom}_{\GF}
(\De((q-1)\rho+q\la),N\otimes\na(\la))\\
=&\sum_{\la\in X^+}
[N\otimes\na(\la):\na((q-1)\rho+q\la)].
\end{align*}
The proposition is proved.
\end{proof}


\section{The Steinberg module multiplicity in a tensor product}
\label{sec:Stmult}

In this section, we will compute the multiplicity of the Steinberg
module $\St$ of $GL_n(q)$ in $\sym\otimes \wedge^{\mu'}(V)\otimes
{\rm Det}^k$ and $\sym\otimes L(\mu)\otimes {\rm Det}^k$, for $0\leq
k\leq q-2$ and certain partitions $\mu=(\mu_1,\ldots,\mu_n)$ of
length $\le n$.


\subsection{Some notations}\label{sec:notations}

We introduce the following notations:
\begin{align*}
\var_i &=  (0,\ldots,0,1,0,\ldots,0)\in X,   \\
\om_m &=  (1^m)=(1,\ldots,1,0,\ldots,0)\in X^+, \quad 1\leq m\leq n.
\end{align*}
For $\la\in X^+$, set $|\la|=\sum_i\la_i$. We denote by
$L(\la)_{\C}$ the irreducible representation of highest weight $\la$
of the general linear Lie algebra $\mathfrak{gl}_n(\C)$, and $[M:
L(\la)_{\C}]$ the multiplicity of $L(\la)_{\C}$ in a semisimple
$\mathfrak{gl}_n(\C)$-module $M$.

Let $\Z_+$ denote the set of nonnegative integers.  We denote by
$e_k(x_1,\ldots, x_n)$ the $k$-th elementary symmetric polynomial
for $k \in \Z_+$,  and denote by $e_\nu(x_1,\ldots, x_n)$ the
elementary symmetric polynomial associated to a partition $\nu$
whose first part $\nu_1 \leq n$.  For a partition $\mu$, define the length $\ell(\mu)$ to be
the number of nonzero parts in $\mu$.
We denote by $m_\mu(x_1,\ldots,
x_n)$  the monomial symmetric polynomial associated to a partition
$\mu$ with $\ell(\mu)\leq n$. Denote by $\mu'$ the conjugate partition of $\mu$.

For a formal series $f(t)\in\Z[[t]]$,
denote by $[t^a]  f(t)$ the coefficient of $t^a$
in $f(t)$ for $a\in\Z_+$.

Let $\Det$ denote the one-dimensional determinant $\Gq$-module. Note
that $\Det\cong  \wedge^n(V)\cong\na(\om_n)$, and that $\Det^{q-1}$
is the trivial module.

For a graded vector space $N^{\bullet}
=\oplus_{i}N^{i}$, we define its Hilbert series to be
$$
H(N^{\bullet};t)=\sum_it^i\dim N^i.
$$
Recall the Steinberg module $\St$ of $\Gq$ is absolutely irreducible
and projective. For a graded $GL_n(q)$-module $N^{\bullet}
=\oplus_{i}N^{i}$, let us denote by
\begin{align*}
H_{\St}(N^{\bullet};t)
=H(\text{Hom}_{GL_n(q)}(\St, N^{\bullet});t)
=\sum_{i}  t^i \dim
\text{Hom}_{GL_n(q)} (\St, N^{i})
\end{align*}
the graded multiplicity of $\St$ in $N^\bullet$.
Similarly, we denote by $H_{\St}(N^{\bullet};t,s)$ (in two variables
$t$ and $s$) the graded multiplicity for $\St$
in a bi-graded $\Gq$-module $N^\bullet=\oplus_{i,j}N^{i,j}$, that is,
$$
H_{\St}(N^{\bullet};t,s)
=\sum_{i,j}  t^i s^j\dim
\text{Hom}_{GL_n(q)} (\St, N^{i,j}).
$$


\subsection{The multiplicity of $\St$ in $S^{\bullet}(V)\otimes \wedge^m(V)\otimes {\Det}^k$.}

We shall use the connection to the algebraic group $\GF$ to
determine the graded multiplicity of the Steinberg module $\St$ in
 $\sym\otimes \wedge^m(V)\otimes \Det^k$.
\begin{thm}\label{thm:symwedg}
Suppose $0\leq m\leq n$. 

(1) If $1\leq k\leq q-2$, then the graded multiplicity of $\St$ in
 $\sym\otimes \wedge^m(V)\otimes \Det^k$ is
\begin{align*}
H_{\St} \big(\sym\otimes \wedge^m(V)\otimes \Det^k;t \big)
=\ds\frac{t^{-n+(q-k)\frac{q^n-1}{q-1}}}{\prod^n_{i=1}(1-t^{q^i-1})}
~e_m(t^{-1},t^{-q},\ldots, t^{-q^{n-1}}).
\end{align*}

(2) 
The graded multiplicity of $\St$
in $\sym\otimes\wedge^m(V)$ is given by
\begin{align*}
&H_{\St} \big(\sym\otimes \wedge^m(V);t \big)\\
&=\ds\frac{t^{-n+\frac{q^n-1}{q-1}}}{\prod^n_{i=1}(1-t^{q^i-1})}
\Big((1-t^{q^n-1})~e_m(t^{-1},t^{-q},\ldots,t^{-q^{n-2}})
+t^{q^n-1}e_m(t^{-1},t^{-q},\ldots,t^{-q^{n-1}})\Big).
\end{align*}
\end{thm}

\begin{proof}
Let us fix $a \in \Z_+$. Observe that $S^a(V)\cong\na(a\om_1)$, and
$\wedge^m(V)\cong\na(\om_m)$. It follows by Lemma~\ref{lem:J.Wang}
that $S^a(V)\otimes\wedge^m(V)\otimes\Det^k$ has a good filtration.
By Proposition~\ref{Pillen}, we turn the problem into a multiplicity
problem in characteristic zero:
\begin{align}
&{\rm dim}~{\rm Hom}_{GL_n(q)}(\St,S^a(V)\otimes\wedge^m(V)\otimes\Det^k)\notag\\
&=\sum_{\la\in X^+} \left[\na(a\om_1)\otimes\na(\om_m)\otimes\na(\om_n)^{\otimes k}
\otimes\na(\la):\na((q-1)\rho+q\la) \right]\notag\\
&=\sum_{\la\in X^+} \big[L(a\om_1)_{\C}\otimes L(\om_m)_{\C}\otimes
L(\om_n)_{\C}^{\otimes k}\otimes L(\la)_{\C}:L((q-1)\rho+q\la)_{\C} \big].
\label{eqn:wedg1}
\end{align}

By applying Pieri's formula twice (cf.
\cite[Proposition~15.25]{FH}), we deduce that
\begin{align*}
&L(a\om_1)_{\C}\otimes L(\om_m)_{\C}\otimes
L(\om_n)_{\C}^{\otimes k}\otimes L(\la)_{\C}\\
&\cong\ds\oplus_{ a_i\in\Z_+, a_1+\cdots+a_n=a,\la_i+k+a_i\leq\la_{i-1}+k}
L(\om_m)_{\C}\otimes L(\la_1+k+a_1,\ldots,\la_n+k+a_n)_{\C}\\
&\cong\oplus
L((\la_1+k+a_1,\ldots,\la_n+k+a_n)+\var_{i_1}+\cdots+\var_{i_m})_{\C},
\end{align*}
where the summation is over the tuples $(a_1,\ldots,a_n)$ and
$(i_1,\ldots,i_m)$ satisfying
\eqref{cond:wedgPieri1}-\eqref{cond:wedgPieri2} below:
\begin{align}
& 1\leq i_1<\cdots<i_m\leq n,\label{cond:wedgPieri1}\\
& a_1  +\ldots +a_n =a, \label{cond:Pieri a}\\
& a_1,  \ldots,a_n\in\Z_+, \label{cond:Pieri b}\\
&  \la_i+k+a_i  \leq\la_{i-1}+k, \label{cond:Pieri c}\\
&  \la+k\om_n+(a_1,\ldots,a_n)+\var_{i_1}+\cdots
 +\var_{i_m}\in X^+.\label{cond:wedgPieri2}
\end{align}

Hence, the multiplicity $[L(a\om_1)_{\C}\otimes L(\om_m)_{\C}\otimes
L(\om_n)_{\C}^{\otimes k}\otimes
L(\la)_{\C}:L((q-1)\rho+q\la)_{\C}]$ is the same as the number of the
tuples $(a_1,\ldots,a_n)$ and $(i_1,\ldots,i_m)$ which satisfy
(\ref{cond:wedgPieri1})--(\ref{cond:Pieri c}) and the following
additional equation
\begin{equation}  \label{eq:weight}
\la+k\om_n
+(a_1,\ldots,a_n)+\var_{i_1}+\ldots+\var_{i_m}=(q-1)\rho+q\la.
\end{equation}
Note that  (\ref{cond:wedgPieri2}) is implied by \eqref{eq:weight}.
Regarding \eqref{eq:weight} as a defining equation for
$(a_1,\ldots,a_n)$, we are reduced to counting the cardinality of
the set  $\Ga^{m}_{\la}$ which consists of the tuples
$(i_1,\ldots,i_m)$ satisfying (\ref{cond:wedgPieri1}) and the
following additional conditions
\eqref{cond:wedg1}-\eqref{cond:wedg5} (recall from \eqref{eq:rhoi}
the notation $\rho_i$):
\begin{align}
&(q-1)|\rho|+(q-1)|\la|-nk-m=a, \label{cond:wedg1}\\
&(q-1)\rho_i+(q-1)\la_i-k\geq 0, \quad \text{ for } i\neq i_1,\ldots,i_m,
 \label{cond:wedg2}\\
&(q-1)\rho_i+(q-1)\la_i-k-1\geq 0, \quad \text{ for }  i= i_1,\ldots,i_m,
 \label{cond:wedg3}\\
&(q-1)\rho_i+q\la_i\leq\la_{i-1}+k, \quad \text{ for } i\neq i_1,\ldots,i_m,
  \label{cond:wedg4}\\
&(q-1)\rho_i+q\la_i-1\leq\la_{i-1}+k, \quad \text{ for }  i=i_1,\ldots,i_m, 2\leq i\leq n.
   \label{cond:wedg5}
\end{align}
Note here that \eqref{cond:Pieri a} gives rise to
\eqref{cond:wedg1}, \eqref{cond:Pieri b} gives rise to
\eqref{cond:wedg2} and \eqref{cond:wedg3}, while  \eqref{cond:Pieri
c} gives rise to \eqref{cond:wedg4} and \eqref{cond:wedg5}.
Hence, we have
\begin{align}
\sum_{\la\in X^+} & \big[L(a\om_1)_{\C}\otimes L(\om_m)_{\C}
\otimes L(\la)_{\C}:L((q-1)\rho+q\la)_{\C} \big]\notag\\
&=\sum_{\la\in X^+}\# \Ga^{m}_{\la} \notag\\
&=\#\{(\la,(i_1,\ldots,i_m))~|~
\la\in X^+,(i_1,\ldots,i_m)\in\Ga^{m}_{\la}\}\notag\\
&=\sum_{1\leq i_1<\cdots<i_m\leq n}
\#\{\la~|~\la\in X^+\text{ satisfies }
(\ref{cond:wedg1})\text{-}(\ref{cond:wedg5})\}.
 \label{eqn:wedg2}
\end{align}

(1) Suppose $1\leq k\leq q-2$. For fixed $1\leq
i_1<i_2<\ldots<i_m\leq n$, let us examine the conditions
(\ref{cond:wedg1})-(\ref{cond:wedg5}) closely. It follows by
\eqref{cond:wedg2}-\eqref{cond:wedg3} for $i=n$ that $\la_n \ge 1$
and hence $\la_i \ge 1$ for all $i=1, \ldots, n$ for $\la \in X^+$;
Moreover, the inequalities $\la_i \ge 1$ for all $i$ guarantee the
validity of \eqref{cond:wedg2}-\eqref{cond:wedg3} in general. Set
\begin{equation*}
\bar{\la}_n=\la_n-1
\end{equation*}
and set,
for $2\leq i\leq n$,
\begin{align*}
\bar{\la}_{i-1}=
\left\{\begin{array}{cc}
\la_{i-1}+k-((q-1)\rho_i+q\la_i-1), \text{ if } i=i_1,\ldots,i_m\\
\la_{i-1}+k-((q-1)\rho_i+q\la_i), \text{ otherwise}.
\end{array}
\right.
\end{align*}
Then the conditions~(\ref{cond:wedg2})-(\ref{cond:wedg5}) hold for
$\la \in X^+$ if and only if
$(\bar{\la}_1,\ldots,\bar{\la}_n)\in\Z_+^n$. By a direct
computation, one further checks that the condition
\eqref{cond:wedg1} is reformulated in terms of the $\bar{\la}_i$'s
as
\begin{equation}  \label{eq:bar la}
a=-n+(q-k) \sum^{n}_{i=1}q^{i-1}-\sum^m_{j=1}
q^{i_j-1}+\sum^n_{i=1}(q^i-1)\bar{\la}_i.
\end{equation}

Hence the equation (\ref{eqn:wedg2}) can be rewritten as
\begin{align*}
\sum_{\la\in X^+} &
\big[L(a\om_1)_{\C}\otimes
L(\om_m)_{\C}\otimes L(\om_n)_{\C}^{\otimes k}
\otimes L(\la)_{\C}:L((q-1)\rho+q\la)_{\C} \big]\\
&=\sum_{1\leq i_1<\cdots<i_m\leq n} \#~\Big
\{(\bar{\la}_1,\ldots,\bar{\la}_n)\in\Z_+^n \text{ that satisfy }
 \eqref{eq:bar la}
 \Big \}\notag\\
&=[t^a] \, \frac{t^{-n+(q-k)\frac{q^n-1}{q-1}}}{\prod^n_{i=1}(1-t^{q^i-1})}
e_m(t^{-1},t^{-q},\ldots,t^{-q^{n-1}}).
\end{align*}
This together with~(\ref{eqn:wedg1}) implies Part (1) of the theorem.

(2) Suppose $k=q-1$. (Note that $\Det^0 \cong \Det^{q-1}$, and we
could also prove Part (2) of the theorem by arguing using $k=0$).
The argument here is similar to (1) above, and so it will be
sketchy. For fixed $1\leq i_1<i_2<\ldots<i_m\leq n$, we set
\begin{align*}
\bar{\la}_n=
\left\{\begin{array}{l}
\la_n-1, \text{ if } i_m\leq n-1\\
\la_n-2, \text{ if } i_m=n.
\end{array}
\right.
\end{align*}
and set, for $2\leq i\leq n$,
\begin{align*}
\bar{\la}_{i-1}=
\left\{\begin{array}{l}
\la_{i-1}+(q-1)-((q-1)\rho_i+q\la_i-1), \text{ if } i=i_1,\ldots,i_m\\
\la_{i-1}+(q-1)-((q-1)\rho_i+q\la_i), \text{ otherwise}.
\end{array}
\right.
\end{align*}
Again it can be verified as before that the
conditions~(\ref{cond:wedg2})-(\ref{cond:wedg5}) hold if and only if
$(\bar{\la}_1,\ldots,\bar{\la}_n)\in\Z_+^n$. Then (\ref{eqn:wedg2})
can be rewritten as
{\allowdisplaybreaks
\begin{align*}
&\sum_{\la\in X^+}[L(a\om_1)_{\C}\otimes L(\om_m)_{\C}
\otimes L(\la)_{\C}:L((q-1)\rho+q\la)_{\C}]\\
&=\sum_{1\leq i_1<\cdots<i_m\leq n-1}
\#~\Big\{(\bar{\la}_1,\ldots,\bar{\la}_n)\in\Z_+^n~\Big|~a
 =-n+\sum^{n}_{i=1}q^{i-1}-\sum^m_{j=1} q^{i_j-1}
 +\sum^n_{i=1}(q^i-1)\bar{\la}_i\Big\}\notag\\
&+
\sum_{1\leq i_1<\cdots<i_m=n}
\#~\Big\{(\bar{\la}_1,\ldots,\bar{\la}_n)\in\Z_+^n~\Big|~a=-n+\sum^{n}_{i=1}q^{i-1}+(q^n-1)\\
&\qquad\qquad\qquad\qquad\qquad\qquad\qquad\qquad\qquad
-\sum^m_{j=1} q^{i_j-1}+\sum^n_{i=1}(q^i-1)\bar{\la}_i\Big\}\notag\\
&=[t^a] \frac{t^{-n+\frac{q^n-1}{q-1}}}{\prod^n_{i=1}(1-t^{q^i-1})}
e_m(t^{-1},t^{-q},\ldots,t^{-q^{n-2}})\\
&+[t^a] \frac{t^{-n+\frac{q^n-1}{q-1}}}{\prod^n_{i=1}(1-t^{q^i-1})}t^{q^n-1}
( e_{m}(t^{-1},t^{-q},\ldots,t^{-q^{n-2}},t^{-q^{n-1}})-e_m(t^{-1},t^{-q},\ldots,t^{-q^{n-2}}))\\
&=[t^a] \ds \frac{t^{-n+\frac{q^n-1}{q-1}}}{\prod^n_{i=1}(1-t^{q^i-1})}
\Big((1-t^{q^n-1})~e_m(t^{-1},t^{-q},\ldots,t^{-q^{n-2}})\\
&\qquad\qquad\qquad\qquad\qquad\qquad\qquad\qquad
+t^{q^n-1}~e_m(t^{-1},t^{-q},\ldots,t^{-q^{n-1}})\Big).
\end{align*}
  }
Therefore together with (\ref{eqn:wedg1}) we have proved Part~(2) of the theorem.
\end{proof}

\begin{rem}
Observe that $\sym\otimes\wedg\otimes \Det^k$ is naturally a
bi-graded $GL_n(q)$-module. Theorem~\ref{thm:symwedg}(2) can be
converted into a formula for the bi-graded multiplicity of $\St$ in
$\sym\otimes\wedg$ as follows:
\begin{align*}
H_{\St} &\big(\sym\otimes\wedg;t,s\big)   \\
&= \frac{t^{-n+\frac{q^n-1}{q-1}}}{\prod^n_{i=1}(1-t^{q^i-1})}
\Big((1-t^{q^n-1}) \prod^{n-2}_{i=0}(1+ st^{q^{-i}})
+t^{q^n-1}  \prod^{n-1}_{i=0}(1+ st^{q^{-i}})   \Big)
  \\
&=  \frac{t^{-n}
(st^{q^n-1}+t^{q^{n-1}})
\prod^{n-2}_{i=0}(s+t^{q^i})}{\prod^n_{i=1}(1-t^{q^i-1})}.
\end{align*}
Similarly (and more easily),  Theorem~\ref{thm:symwedg}(1) is
converted into the following formula for $1\le k \le q-2$:
\begin{align*}
H_{\St}(\sym\otimes\wedg\otimes\Det^k;t,s)
=  t^{-n+(q-1-k)\frac{q^n-1}{q-1}} \cdot
\frac{\prod^{n-1}_{i=0}(s+t^{q^i})}{\prod^n_{i=1}(1-t^{q^i-1})}.
\end{align*}
In this way, we obtain a new proof of \cite[Theorem~C]{WW}, which
was in turn a generalization of the earlier work \cite{Mi, MP}
(where $q$ is assumed to be a prime).
\end{rem}
There is an isomorphism of $GL_n(q)$-modules $\wedge^{n-m}(V)\cong
\wedge^{m}(V)^*\otimes\wedge^n(V)$ and $\wedge^n(V)\cong \Det$,
where $W^*$ denotes the dual module of a $\Gq$-module $W$. Hence, we
have an isomorphism of $GL_n(q)$-modules
\begin{equation}\label{eqn:wedgdual}
\wedge^m(V)^* \cong \wedge^{n-m}(V)\otimes\Det^{-1}\cong
\wedge^{n-m}(V)\otimes\Det^{q-2}, \qquad 0\leq m\leq n.
\end{equation}
Theorem~\ref{thm:symwedg} can now be converted into the following
form using \eqref{eqn:wedgdual}.

\begin{cor}
The graded multiplicity of $\St$ in
$\sym\otimes \wedge^m(V)^*\otimes \Det^k$ is given by
\begin{align*}
&H_{\St} \big(\sym\otimes \wedge^m(V)^*\otimes \Det^k;t\big)\\
=&\left\{
\begin{array}{ll}
\ds\frac{t^{-n+\frac{q^n-1}{q-1}}}{\prod^n_{i=1}(1-t^{q^i-1})}
\Big((1-t^{q^n-1})~e_{n-m}(t^{-1},t^{-q},\ldots,t^{-q^{n-2}})&\\
\qquad\qquad\qquad\qquad\qquad\quad+t^{q^n-1}~e_{n-m}(t^{-1},t^{-q},
\ldots,t^{-q^{n-1}})\Big),&  \text{ if } \/k=1\\
\ds\frac{t^{-n+(q+1-k)\frac{q^n-1}{q-1}}}{\prod^n_{i=1}(1-t^{q^i-1})}
~e_{n-m}(t^{-1},t^{-q},\ldots,t^{-q^{n-2}},t^{-q^{n-1}}),& \text{ if
} \/2 \le k \le q-1.
\end{array}
\right.
\end{align*}
\end{cor}

\subsection{The multiplicity of $\St$ in $\sym\otimes \wedge^\nu(V)\otimes\Det^k$.}

For a partition $\nu=(\nu_1,\ldots,\nu_{\ell})$ with $\nu_1\leq n$,
denote by $\wedge^{\nu}(V)$ the $GL_n(q)$-module
$$
\wedge^{\nu}(V)=\wedge^{\nu_1}(V)\otimes\cdots\otimes\wedge^{\nu_{\ell}}(V).
$$
Recall $e_\nu$ denotes the elementary symmetric polynomial
associated to $\nu$. The following  is a generalization of
Theorem~\ref{thm:symwedg}(1) (which corresponds to the case $\ell
=1$ below).

\begin{thm}\label{symwedgnu}
Let  $\nu=(\nu_1,\ldots,\nu_{\ell})$ be a partition with length
$\ell(\nu)=\ell$ and $\nu_1\leq n$. Let $k$ be a positive integer
such that  $k+\ell\leq q-1$. Then the graded multiplicity of $\St$
in $\sym\otimes \wedge^{\nu}(V)\otimes \Det^k$ is given by
\begin{align*}
H_{\St} \big(\sym\otimes \wedge^{\nu}(V)\otimes \Det^k;t\big)
=\ds\frac{t^{-n+(q-k)\frac{q^n-1}{q-1}}}{\prod^n_{i=1}(1-t^{q^i-1})}
~e_{\nu}(t^{-1},t^{-q},\ldots,t^{-q^{n-1}}).
\end{align*}
\end{thm}
\begin{proof}
Fix $a\in \Z_+$. By arguments similar to Theorem~\ref{thm:symwedg},
one can show using Proposition~\ref{Pillen} and Pieri's formula
that, for $a\geq 0$,
\begin{align}
&{\rm dim~Hom}_{GL_n(q)}(St,~ S^a(V)\otimes\wedge^{\nu}(V)\otimes\Det^k)\notag\\
&=\sum_{\la\in X^+} \big[ L(a\om_1)_{\C}\otimes
L(\om_{\nu_1})_{\C}\otimes\cdots\otimes L(\om_{\nu_{\ell}})_{\C}\otimes
L(\om_n)_{\C}^{\otimes k}\otimes L(\la)_{\C}:L((q-1)\rho+q\la)_{\C} \big]\notag\\
&=\sum_{\la\in X^+}|\Ga^{\nu}_{\la}|, \label{eqn:wedgnu1}
\end{align}
where $\Ga^{\nu}_{\la}$ is the set consisting of the sequences
$\left((a_1,\ldots,a_n), (i^b_u | 1\leq u\leq\nu_b, 1\leq
b\leq\ell)\right)$ satisfying ~(\ref{cond:Pieri a})-(\ref{cond:Pieri
c}) and the following additional conditions
\eqref{cond:nuPieri1}-\eqref{cond:nuPieri3}:
\begin{align}
&1\leq i^j_1<\cdots<i^j_{\nu_j} \leq n, \quad 1\leq j\leq \ell,\label{cond:nuPieri1}\\
& \la+(a_1,\ldots,a_n)+k\om_n + \sum_{b=1}^j \sum_{u=1}^{\nu_b}
\var_{i_u^b} \in X^+, \quad 1\leq j\leq \ell,\label{cond:nuPieri2}\\
& \la+(a_1,\ldots,a_n)+k\om_n
+\sum_{b=1}^\ell \sum_{u=1}^{\nu_b} \var_{i_u^b} =(q-1)\rho+q\la.\label{cond:nuPieri3}
\end{align}
Note that (\ref{cond:nuPieri2}) will automatically hold if
$\left((a_1,\ldots,a_n), (i^b_u | 1\leq u\leq\nu_b, 1\leq
b\leq\ell)\right)$ satisfies (\ref{cond:nuPieri1})~
and~(\ref{cond:nuPieri3}). This is clear once we visualize the
weight $(q-1)\rho+q\la$ as a Young diagram whose consecutive rows
differ by at least $(q-1)$, and removing at most $\ell$ boxes in
each row gives rise to  new Young diagrams  corresponding to the
weights  in (\ref{cond:nuPieri2}) (recall here $\ell \le q-1$ by
assumption).

For $1\leq i\leq n$, denote
\begin{equation*}
c_i=\#\{(b,u)~|~i^b_u=i, 1\leq u\leq\nu_b, 1\leq b\leq\ell \}.
\end{equation*}
Hence, regarding \eqref{cond:nuPieri3} as a defining relation for
$(a_1,\ldots,a_n)$, we see that the set $\Ga^{\nu}_{\la}$ has the
same cardinality as the set consisting of the sequences $(i^b_u |
1\leq u\leq\nu_b, 1\leq b\leq\ell)$ satisfying~(\ref{cond:nuPieri1})
and \eqref{cond:wedgnu1}-\eqref{cond:wedgnu3} below:
\begin{align}
& (q-1)\rho_i+(q-1)\la_i-k-c_i \geq 0,  \quad1\leq i\leq n,   \label{cond:wedgnu1}\\
& (q-1)\rho_i+q\la_i-c_i \leq\la_{i-1}+k,  \quad  2\leq i\leq n,  \label{cond:wedgnu2} \\
& (q-1)|\rho|+(q-1)|\la|-nk-|\nu| =a.  \label{cond:wedgnu3}
\end{align}
Therefore, the identity (\ref{eqn:wedgnu1}) can be rewritten as
\begin{align}
&{\rm dim~Hom}_{GL_n(q)}(\St,~ S^a(V)\otimes\wedge^{\nu}(V)\otimes\Det^k)\notag\\
=&\sum_{\la\in X^+}
\#
\Big\{\text{the tuples } \big(i^b_u | 1\leq u\leq\nu_b, 1\leq b\leq\ell\big)
 \text{ that satisfy }(\ref{cond:nuPieri1}),
  (\ref{cond:wedgnu1})\text{--}(\ref{cond:wedgnu3})\Big\}\notag\\
=&\sum_{1\leq i^j_1<\cdots<i^j_{\nu_j}\leq n, ~1\leq j\leq \ell}
\#\left\{\la\in X^+\text{ that satisfy
}(\ref{cond:wedgnu1})\text{--}(\ref{cond:wedgnu3}) \right\}.
\label{eqn:wedgnu2}
\end{align}
Given $1\leq i^j_1<\cdots<i^j_{\nu_j}\leq n$ for $1\leq j\leq \ell$,
we set $\bar{\la}=(\bar{\la}_1,\ldots,\bar{\la}_n)$ for $\la\in X^+$ with
\begin{align*}
\bar{\la}_n&=\la_n-1,\\
\bar{\la}_{i-1}&=\la_{i-1}+k-
((q-1)\rho_i+q\la_i-c_i), \quad 2\leq i\leq n.
\end{align*}

We claim that~(\ref{cond:wedgnu1})~and~(\ref{cond:wedgnu2}) hold for
$\la \in X^+$ if and only if
$\bar{\la}=(\bar{\la}_1,\ldots,\bar{\la}_n)\in\Z_+^n$. In fact,
(\ref{cond:wedgnu2})~ is clearly equivalent to that
$\bar{\la}_i\geq0$ for $1\leq i\leq n-1$. Also,
\eqref{cond:wedgnu1} for $i=n$ (which reads that $(q-1)\la_n-k-
c_n\geq0$) holds if and only if $\bar{\la}_n=\la_n-1\geq 0$, since
$1\leq k+c_n\leq k+\ell\leq q-1$ by assumption. It is further
readily checked that $\la \in X^+$ follows from $\bar{\la} \in
\Z_+^n$. So it remains to see that if
$(\bar{\la}_1,\ldots,\bar{\la}_n)\in\Z_+^n$ then
$
(q-1)\rho_i+(q-1)\la_i-k-c_i\geq0
$
for $1\leq i\leq n-1$;
this follows from the facts that
$
k+c_i\leq k+\ell\leq q-1
$
and $\la_i\geq\la_n\geq 1$.

On the other hand, a direct calculation shows that
\begin{align*}
\sum^n_{i=1}(q^i-1)\bar{\la}_i
=&\sum^n_{i=1}(q-1)\la_i-(q^n-1)
+\sum^{n}_{i=1}k(q^{i-1}-1)\\
&-\sum^{n}_{i=1}(q^{i-1}-1)(q-1)\rho_{i}
+\sum^{n}_{i=1}(q^{i-1}-1)c_i.
\end{align*}
It follows that
{\allowdisplaybreaks
\begin{align*}
(q-1)|\rho|& +(q-1)|\la|-nk-|\nu|\\
=& \sum^n_{i=1}(q^i-1)\bar{\la}_i+(q^n-1)+\sum^{n}_{i=1}q^{i-1} (q-1)\rho_{i}\\
&\quad -\sum^{n}_{i=1}k(q^{i-1}-1)-\sum^{n}_{i=1}(q^{i-1}-1)c_i
-nk-|\nu|\\
=&\sum^n_{i=1}(q^i-1)\bar{\la}_i+(q^n-1)+(-n+1+q+\cdots+q^{n-1})\\
&\quad -\sum^{n}_{i=1}kq^{i-1}
-\sum^{n}_{i=1}q^{i-1}c_i+\sum^{n}_{i=1}c_i -|\nu|\\
=&\sum^n_{i=1}(q^i-1)\bar{\la}_i+(-n) +(q-k)\frac{q^n-1}{q-1}-\sum^n_{i=1}q^{i-1}c_i,
\end{align*}
  }
  since $\sum^{n}_{i=1}c_i=\nu_1+\ldots+\nu_{\ell}=|\nu|$.
This together with the identity
\begin{align*}
\sum^n_{i=1}q^{i-1}c_i=\sum_{1\leq b\leq \ell, 1\leq u\leq \nu_b}q^{i^b_u-1}
\end{align*}
implies that~(\ref{cond:wedgnu3}) is equivalent to
\begin{align}
a=\sum^n_{i=1}(q^i-1)\bar{\la}_i+(-n)+(q-k)\frac{q^n-1}{q-1}
-\sum_{1\leq b\leq \ell, 1\leq u\leq \nu_b}q^{i^b_u-1}.
\label{eqn:wedgnu3}
\end{align}

Summarizing, we can rewrite (\ref{eqn:wedgnu2}) as
\begin{align*}
&{\rm dim~Hom}_{GL_n(q)}(\St,~ S^a(V)\otimes\wedge^{\nu}(V)\otimes\Det^k)\\
&=\sum_{1\leq i^j_1<\cdots<i^j_{\nu_j}\leq n, ~1\leq j\leq \ell}
\#\big\{(\bar{\la}_1,\ldots,\bar{\la}_n) \in \Z_+^n
\text{ that satisfy }(\ref{eqn:wedgnu3})\big\}\\
&=[t^a]
\frac{t^{-n+(q-k)\frac{q^n-1}{q-1}}}{\prod^n_{i=1}(1-t^{q^i-1})}
\prod^{\ell}_{j=1}e_{\nu_j}(t^{-1},t^{-q},\ldots,t^{-q^{n-1}})\\
&=[t^a]
\frac{t^{-n+(q-k)\frac{q^n-1}{q-1}}}{\prod^n_{i=1}(1-t^{q^i-1})}
e_{\nu}(t^{-1},t^{-q},\ldots,t^{-q^{n-1}}).
\end{align*}
Therefore,
\begin{align*}
H_{\St} &\big(\sym\otimes \wedge^{\nu}(V)\otimes \Det^k;t \big)  \\
=&\sum_{a\geq 0}\big({\rm dim~Hom}_{GL_n(q)}
(\St, S^a(V)\otimes\wedge^{\nu}(V)\otimes\Det^k)\big)\cdot t^a\\
=&\ds\frac{t^{-n+(q-k)\frac{q^n-1}{q-1}}}{\prod^n_{i=1}(1-t^{q^i-1})}
e_{\nu}(t^{-1},t^{-q},\ldots,t^{-q^{n-1}}).
\end{align*}
The theorem is proved.
\end{proof}
\begin{rem}
Using a similar argument, we can in principle obtain a (very messy in general) formula  for
the graded multiplicity of $\St$ in
$\sym\otimes \wedge^{\nu}(V)\otimes \Det^k$
without the assumption that $k+\ell\leq q-1$, generalizing Theorem~\ref{thm:symwedg}(2).
For example, the formula in the case $\nu=(1,1,\ldots,1)$ is given by
\begin{align*}
H_{\St} & \big(S^\bullet(V)\otimes V^{\otimes \ell}\otimes \Det^k;t \big) \\
=&\ds\frac{t^{-n+(q-k)\frac{q^n-1}{q-1}}}{\prod^n_{i=1}(1-t^{q^i-1})}
\sum^{q-1-k}_{c=0}\binom{\ell}{c}(t^{-1}+t^{-q}+\cdots+t^{-q^{n-2}})^{\ell-c}(t^{-q^{n-1}})^c\\
+&\ds\frac{t^{-n+(2q-1-k)\frac{q^n-1}{q-1}}}{\prod^n_{i=1}(1-t^{q^i-1})}
\sum^{\ell}_{c=q-k}\binom{\ell}{c}(t^{-1}+t^{-q}+\cdots+t^{-q^{n-2}})^{\ell-c}(t^{-q^{n-1}})^c
\end{align*}
for $1\leq k\leq q-1$.
\end{rem}
By~(\ref{eqn:wedgdual}), one has
\begin{equation}\label{eqn:wedgdualnu}
\wedge^{\nu}(V)^*
\cong \wedge^{n-\nu_{\ell}}(V)
\otimes\cdots\otimes\wedge^{n-\nu_1}(V)
\otimes\Det^{q-1-\ell}
\end{equation}
for any partition $\nu$ with $\nu_1 \le n$. Hence,
Theorem~\ref{symwedgnu} can be converted into the following form and
vice versa.
\begin{cor}\label{cor:symwedgnu*}
Let $\nu=(\nu_1,\ldots,\nu_{\ell})$ be a (nonempty) partition
with length $\ell(\nu)=\ell$ such that $\nu_1\leq n$.
For $\ell< k\leq q-1$, the graded multiplicity of $\St$ in
$\sym\otimes \wedge^{\nu}(V)^*\otimes \Det^k$
is given by
\begin{align*}
H_{\St} \big(\sym\otimes \wedge^{\nu}(V)^*\otimes \Det^k;t \big)
=\ds\frac{t^{-n+(q-k)\frac{q^n-1}{q-1}}}{\prod^n_{i=1}(1-t^{q^i-1})}
e_{\nu}(t,t^{q},\ldots,t^{q^{n-1}}).
\end{align*}
\end{cor}


\subsection{The multiplicity of $\St$ in $\sym\otimes L(\mu)\otimes\Det^k$}

Let $S_d$ be the symmetric group of $d$ letters. According to the
Schur duality, $V^{\otimes d}$ as a $\GF \otimes S_d$-module (and
hence also as a $\Gq \otimes S_d$-module) is semisimple and
multiplicity-free, for $1\leq d\leq p-1$. We now compute the
multiplicity of $\St$ in $\sym\otimes L(\mu)\otimes\Det^k$. Our more
restrictive condition on $\mu$ below ensures that $L(\mu)$ appears
as a direct summand in $V^{\otimes d}$ for $1\leq d\leq p-1$.

\begin{thm}\label{thm:symirr}
Fix $1\leq k\leq q-1$ and $1\leq d\leq p-1$.
 Let $\mu=(\mu_1,\ldots,\mu_n)$ be a partition of
$d$ with $\ell(\mu)\leq n$.

(1)  If $k+\mu_1\leq q-1$,
then the graded multiplicity of $\St$ in the graded $GL_n(q)$-module
$\sym\otimes L(\mu)\otimes \Det^k$
is given by
\begin{align*}
H_{\St} \big(\sym\otimes L(\mu)\otimes \Det^k;t \big)
=\ds\frac{t^{-n+(q-k)\frac{q^n-1}{q-1}}}{\prod^n_{i=1}(1-t^{q^i-1})}
s_{\mu}(t^{-1},t^{-q},\ldots,t^{-q^{n-1}}).
\end{align*}

(2) If $\mu_1<k\leq q-1$, then the graded multiplicity of $\St$ in
the graded $GL_n(q)$-module $\sym\otimes L(\mu)^*\otimes \Det^k$ is
given by
\begin{align*}
H_{\St} \big(\sym\otimes L(\mu)^*\otimes \Det^k;t \big)
=\ds\frac{t^{-n+(q-k)\frac{q^n-1}{q-1}}}{\prod^n_{i=1}(1-t^{q^i-1})}
s_{\mu}(t,t^{q},\ldots,t^{q^{n-1}}).
\end{align*}
\end{thm}
\begin{proof}
(1) Suppose  $k+\mu_1\leq q-1$ and let  $\mu'$  be the conjugate of
$\mu$. Then $\ell(\mu')=\mu_1$ and hence $k+\ell(\mu')\leq q-1$. By
Theorem~\ref{symwedgnu}, we have
\begin{align*}
H_{\St} \big(\sym\otimes\wedge^{\mu'}(V)\otimes\Det^k;t \big)
=\ds\frac{t^{-n+(q-k)\frac{q^n-1}{q-1}}}{\prod^n_{i=1}(1-t^{q^i-1})}
e_{\mu'}(t^{-1},t^{-q},\ldots,t^{-q^{n-1}}).
\end{align*}
Since $d \le p-1$, the $\Gq$-module $V^{\otimes d}$ is semisimple
and so is its submodule  $\wedge^{\mu'}(V)$. Hence, we have the
following decomposition of $\Gq$-modules:
\begin{align}
\wedge^{\mu'}(V)
\cong\sum_{\ell(\ga)\leq n,\ga\leq\mu}K_{\ga'\mu'}L(\ga)
\cong L(\mu)\oplus\sum_{\ell(\ga)\leq n,\ga<\mu}K_{\ga'\mu'}L(\ga),
\label{eqn:SWduality}
\end{align}
where $K_{\ga'\mu'}$ are the Kostka numbers (cf. e.g. \cite{FH}).

We complete the proof of Part~(1) by induction on the dominance
order of $\mu$. For $\mu$ minimal in the sense that there is no
partition $\ga <\mu$ with $\ell(\ga) \le n$, Part~(1) reduces to
Theorem~\ref{symwedgnu}.

Now let $\mu=(\mu_1,\ldots,\mu_n)$ be a general partition of $d$
with $\ell(\mu)\leq n$. For $\ga <\mu$  with $\ell(\ga) \le n$, we
have $\ga_1 \leq \mu_1$ and  $k+\ga_1\leq q-1$. Thus by induction
hyperthesis we have
\begin{align*}
H_{\St}(\sym\otimes L(\ga)\otimes \Det^k;t)
=\ds\frac{t^{-n+(q-k)\frac{q^n-1}{q-1}}}{\prod^n_{i=1}(1-t^{q^i-1})}
s_{\ga}(t^{-1},t^{-q},\ldots,t^{-q^{n-1}}).
\end{align*}
This together with~(\ref{eqn:SWduality}) gives us
\begin{align*}
H&_{\St} \big(\sym\otimes L(\mu)\otimes\Det^k;t \big)\\
=&H_{\St} \big(\sym\otimes \wedge^{\mu'}(V)\otimes\Det^k;t\big)
-\sum_{\ell(\ga)\leq n,\ga<\mu}K_{\ga'\mu'}H_{\St}
\big(\sym\otimes L(\ga)\otimes\Det^k;t \big)\\
=&\ds\frac{t^{-n+(q-k)\frac{q^n-1}{q-1}}}{\prod^n_{i=1}(1-t^{q^i-1})}~
\Big(e_{\mu'}(t^{-1},t^{-q},\ldots,t^{-q^{n-1}})
-\sum_{\ell(\ga)\leq n,\ga<\mu}K_{\ga'\mu'}
s_{\ga}(t^{-1},t^{-q},\ldots,t^{-q^{n-1}})\Big)\\
=&\ds\frac{t^{-n+(q-k)\frac{q^n-1}{q-1}}}
{\prod^n_{i=1}(1-t^{q^i-1})}~s_{\mu}(t^{-1},t^{-q},\ldots,t^{-q^{n-1}}),
\end{align*}
where the last equality is due to the symmetric function identity
$$e_{\mu'} (x_1, x_2,\ldots, x_n)
=\sum_{\ga\vdash d, \ell(\ga)\leq n,\ga\leq\mu}K_{\ga'\mu'}s_{\ga} (x_1, x_2,\ldots, x_n).$$

(2) Suppose $\mu_1<k\leq q-1$ and let $\mu'$ be the conjugate of
$\mu$. Then $\ell(\mu')=\mu_1$ and hence $\ell(\mu')<k$. By
Corollary~\ref{cor:symwedgnu*} we have
\begin{align*}
H_{\St}(\sym\otimes\wedge^{\mu'}(V)^*\otimes\Det^k;t)
=\ds\frac{t^{-n+(q-k)\frac{q^n-1}{q-1}}}
{\prod^n_{i=1}(1-t^{q^i-1})}~e_{\mu'}(t,t^{q},\ldots,t^{q^{n-1}}).
\end{align*}
On the other hand, by ~(\ref{eqn:SWduality}) one has
\begin{align*}
\wedge^{\mu'}(V)^*
\cong\sum_{\ga\vdash d, \ell(\ga)\leq n,\ga\leq\mu}K_{\ga'\mu'}L(\ga)^*
\cong L(\mu)^*\oplus\sum_{\ga\vdash d, \ell(\ga)\leq n,\ga<\mu}K_{\ga'\mu'}L(\ga)^*,
\label{eqn:SWduality}
\end{align*}
The theorem follows by an argument similar to the proof of
Part~(1).
\end{proof}


\section{The composition multiplicity in $\sym$ around the Steinberg module}
\label{sec:compmult}

In this section, we shall determine the graded multiplicity in the
symmetric algebra $\sym$ for a large class of of irreducible modules
around the Steinberg module.

Recall the definition of the Hilbert series of a graded vector space from Section \ref{sec:notations}.
Denote by $P(\la)$ the projective cover of the irreducible
$\Gq$-module $L(\la)$ for $\la\in X_r$. The graded multiplicity of
$L(\la)$ in the symmetric algebra $\sym$ is equal to the Hilbert
series of the graded space ${\rm Hom}_{\Gq}(P(\la), \sym)$.


\subsection{A case for any prime $p$}

Recall that $q=p^r$.
 For $\la\in X_r$,  set
 \begin{equation}
 \la^0=(q-1)\rho+w_0\la.
 \label{eqn:wt0}
 \end{equation}
According to Tsushima \cite{Ts} (which goes back to Lusztig
\cite{Lu} for $m=1$ and $p>2$),  we have, for $1\leq m\leq n-1$,
that
\begin{align*}
\ds\St\otimes\wedge^m(V)
\cong\left\{
\begin{array}{ll}
P(\om_m^0),&\text{ if }q>2 \\
P(\om_m^0)\oplus\St,&\text{ if }q=2,
\end{array}
\right.
\end{align*}
and hence also, for all $1\leq k\leq q-1$,
\begin{align}\label{eqn:Ts}
\ds\St\otimes\wedge^m(V)\otimes\Det^k
\cong\left\{
\begin{array}{ll}
P(\om_m^0+k\om_n),&\text{ if }q>2  \\
P(\om_m^0+k\om_n)\oplus\St\otimes\Det^k,&\text{ if }q=2.
\end{array}
\right.
\end{align}

Recall the Kronecker symbol $\delta_{2,q}=1$ if $q=2$ and
$\delta_{2,q}=0$ if $q>2$. For $1\leq m\leq n-1$, define
$\ga=(\ga_1,\ldots,\ga_n)\in X_r$ by $\ga_i=(q-1)(n-i)-k-1$ for
$1\leq i\leq m$ and $\ga_i=(q-1)(n-i)-k$ for $m+1\leq i\leq n$.

\begin{thm}\label{thm:grcompmult1}
Suppose $1\leq m\leq n-1$ and $1\leq k\leq q-1$.

(1) For $1\leq k\leq q-2$,  the graded composition multiplicity of
$L(\ga)$ in $S^\bullet(V)$ is given by
\begin{align*}
\ds\frac{t^{-n+(q-k)\frac{q^n-1}{q-1}}}{\prod^n_{i=1}(1-t^{q^i-1})}~
e_m(t^{-1},t^{-q},\ldots,t^{-q^{n-2}},t^{-q^{n-1}}).
\end{align*}
(2) The graded composition multiplicity of $L(\ga)$  in
$\sym$ is given by
\begin{align*}
\ds&\frac{t^{-n+\frac{q^n-1}{q-1}}}{\prod^n_{i=1}(1-t^{q^i-1})}
\Big((1-t^{q^n-1})~e_m(t^{-1},t^{-q},\ldots,t^{-q^{n-2}})   \\
& \qquad \qquad\qquad \qquad +t^{q^n-1}~e_m(t^{-1},t^{-q},
 \ldots,t^{-q^{n-2}},t^{-q^{n-1}})\Big)
  - \delta_{2,q}\frac{t^{-n+\frac{q^n-1}{q-1}}}{\prod^n_{i=1}(1-t^{q^i-1})}.
\end{align*}
\end{thm}
\begin{proof}
Suppose $1\leq k\leq q-2$; observe that this happens only when $q>2$.

By~(\ref{eqn:wedgdual}) and (\ref{eqn:Ts}) we have
\begin{align*}
\St\otimes(\wedge^m(V)\otimes\Det^k)^*
\cong \St\otimes \wedge^{n-m}(V)\otimes\Det^{-1-k}
\cong P(\om^0_{n-m}-(k+1)\om_n)=P(\ga).
\end{align*}
This implies that
\begin{align*}
{\rm Hom}_{\Gq} & (P(\ga),\sym)\\
\cong&{\rm Hom}_{\Gq}(\St\otimes(\wedge^m(V)\otimes\Det^k)^*,\sym)\\
\cong&{\rm Hom}_{\Gq}(\St,  \sym\otimes\wedge^m(V)\otimes\Det^k).
\end{align*}
Thus,  Part (1) of the theorem follows by
Theorem~\ref{thm:symwedg}(1).  By a similar argument and
Theorem~\ref{thm:symwedg}(2),  the second part in the case $q>2$
follows.

Suppose now  $q=2$. Note that the determinant module $\Det$
coincides with the trivial module.  Using ~(\ref{eqn:Ts}) we get
\begin{align*}
\St\otimes \wedge^m(V)^*
\cong \St\otimes \wedge^{n-m}(V)
\cong P(\ga)\oplus \St
\end{align*}
and hence, for $a\geq0$,
\begin{align*}
\dim  & {\rm Hom}_{\Gq}  (P(\ga),S^a(V))\\
 =& \dim {\rm Hom}_{\Gq}(\St\otimes \wedge^m(V)^*,S^a(V))
- \dim {\rm Hom}_{\Gq}(\St,S^a(V))\\
 =&\dim {\rm Hom}_{\Gq}(\St,  S^a(V)\otimes\wedge^m(V))
-\dim   {\rm Hom}_{\Gq}(\St,S^a(V)).
\end{align*}

The theorem for $q=2$ now follows from this identity and
Theorem~\ref{thm:symwedg}(2). Note a special case of
Theorem~\ref{thm:symwedg}(2) for $m=0$ says that the graded
multiplicity of the Steinberg module $\St$ in the symmetric algebra
is given by $t^{-n+\frac{q^n-1}{q-1}}
\prod^n_{i=1}(1-t^{q^i-1})^{-1}. $
\end{proof}

\begin{rem}
Theorem~\ref{thm:grcompmult1} under the additional assumption that
$q$ is a prime was first established by Carlisle and Walker
\cite[Corollary~1.4]{CW} using a completely different method.
Their formula is equivalent to ours  since 
\begin{align*}
&\frac{t^{-n+(q-k)\frac{q^n-1}{q-1}}}{\prod^n_{i=1}(1-t^{q^i-1})}~
e_m(t^{-1},t^{-q},\ldots, t^{-q^{n-1}})\\
=&\frac{t^{-k\frac{q^n-1}{q-1}}}{\prod^n_{i=1}(1-t^{q^i-1})}~
\sum_{0\leq i_1<i_2<\ldots<i_m\leq n-1}t^{q+q^2+\ldots+q^n-q^{i_1}-\ldots-q^{i_m}-n}.
\end{align*}
\end{rem}
%

\subsection{A general composition multiplicity  formula}

Throughout this subsection, we assume that $p>n$,  the Coxeter
number of $\Gq$.

We denote by
$\text{br}(N)$ the Brauer character of a $\Gq$-module $N$ (cf. \cite{Hu}). For any $W$-invariant
element $\ga \in \Z[X]^W$ with $W =S_n$, one writes it as a linear combination in terms of
the formal characters of the simple modules $L(\la)$ for $\la \in X_r$
and define its Brauer character $\text{Br} (\ga)$ to be the corresponding
linear combination of $\text{Br} (L(\la))$'s.
Denote by $W^{\mu}$
a set of coset representatives in $W$ of the stabilizer subgroup
$W_{\mu}$ of $\mu$. Then the orbit sum
$\sum_{\sigma \in W^\mu} e^{\sigma\mu} \in\Z[X]^W$ has the monomial
symmetric polynomial $m_\mu(e^{\var_1}, \ldots, e^{\var_n})$ as its formal character,
and its Brauer character will be denoted by $\text{Br}(m_\mu)$.
Let $\phi $ denote the
Brauer character of $\St$.
The following is essentially a result of
Ballard \cite[Proposition~7.2]{Ba} in the case of $\Gq$.

\begin{lem} \label{lem:Ballard}
Assume that $p>n$. If $\mu\in X_r$ satisfies $\mu_1-\mu_n\leq p-1$,
then the Brauer character of the projective $\Gq$-module $P(\mu^0)$
is equal to $\phi \text{Br}(m_{\mu})$. In particular,
$$
{\dim} P(\mu^0)= {\dim}\, \St \cdot|W^{\mu}|
=\frac{\prod^{n-1}_{i=0} (q^n-q^i)}{\prod^n_{i=1}(q^i-1)}  m_{\mu} (1,1,\ldots,1).
$$
\end{lem}

\begin{rem} \label{rem:Chastkofsky}
For $q=p^r$ with $r\geq 2$, the bound in \cite[Proposition 7.2]{Ba}
can be stated as $\mu_1-\mu_n\leq p-1$ as above, while for $q=p$,
Ballard imposed the more restrictive assumption that
$\mu_1-\mu_n < (p-1)/2$ (because of using \cite[Lemma 7.3]{Ba}).
Chastkofsky remarked in his math review on \cite{Ba} that the bound
in Ballard's paper can always be improved to $\mu_1-\mu_n\leq p-1$,
using his work  \cite{Ch} (also see Jantzen \cite{J1} for closely
related results).
%
%
\end{rem}

Recall that (cf. \cite{FH}) elementary symmetric functions can be
expressed in terms of monomial symmetric functions as follows:
\begin{equation}\label{eqn:em}
e_{\nu}=\sum_{\mu\leq\nu'}a_{\nu\mu}m_{\mu},
\end{equation}
where $a_{\nu\mu}$ are nonnegative integers and $a_{\nu\nu'}=1$.
This can be seen by expressing $e_\nu$ in terms of Schur functions
$s_\la$ and then $s_\la$ in terms of $m_\mu$.
\begin{prop}\label{decompStIrr}
Suppose  $p>n$ and $0\leq k\leq q-2$. Let $\nu$ be a partition such
that $\nu_1\leq n$  and $\nu'_1-\nu'_n\leq p-1$. Then the
$\Gq$-module $\St\otimes \wedge^{\nu}(V)\otimes\Det^k$ can be
decomposed as
\begin{equation*}
\St\otimes  \wedge^{\nu}(V)\otimes\Det^k
\cong
\bigoplus_{\tau\leq\nu', \ell(\tau) \le n} P(\tau^0+k\om_n)^{\oplus a_{\nu\tau}},
\end{equation*}
where $a_{\nu\tau}$  is defined in (\ref{eqn:em}).
\end{prop}
\begin{proof}
As the module $\St\otimes  \wedge^{\nu}(V)\otimes\Det^k$ is known to
be projective, it suffices to check the isomorphism in the
proposition on the Brauer character level. We can further assume
$k=0$, as the general case is then obtained easily by tensoring by
$\Det^k$. By (\ref{eqn:em}), the Brauer character of $\St\otimes
\wedge^{\nu}(V)$ can be written as
\begin{align*}
{\rm Br} (\St\otimes  \wedge^{\nu}(V) )
=\phi {\rm Br} (\wedge^{\nu}(V) )
=\sum_{\tau\leq\nu', \ell(\tau) \le n}a_{\nu\tau}\phi \text{Br}(m_{\tau}).
\end{align*}
Observe that  if $\tau\leq\nu'$ then
$\tau_1-\tau_n\leq\nu'_1-\nu'_n\leq p-1$. It follows by
Lemma~\ref{lem:Ballard} that $\phi \text{Br}(m_{\tau})$ is the Brauer character of
the projective indecomposable module $P(\tau^0)$. So we have established
the desired identity of Brauer characters.
\end{proof}

The projective $\Gq$-module $P(\la)$ for $\la\in X_r$ has both head
and socle isomorphic to $L(\la)$, and therefore $P(\la)^*\cong
P(-w_0 \la)$. Then thanks to the fact that $\St^* \cong \St$,
Proposition~\ref{decompStIrr} can be converted into the following.

\begin{cor}\label{cor:decompStIrr*}
Suppose  $p>n$ and $0\leq k\leq q-2$. Let $\nu$ be a partition such
that $\nu_1\leq n$  and $\nu'_1-\nu'_n\leq p-1$.
Then the $\Gq$-module $\St\otimes \wedge^{\nu}(V)^*\otimes\Det^k$ can be decomposed as %
\begin{align*}\St\otimes \wedge^{\nu}(V)^*\otimes\Det^k
\cong
\sum_{\tau\leq\nu', \ell(\tau) \le n} P((q-1)\rho-\tau+k\om_n)^{\oplus a_{\nu\tau}}.
\end{align*}
\end{cor}

Below is a main result of this section.
\begin{thm} \label{thm:grcompmult2}
Suppose $p>n$ and $0\leq k\leq q-2$. Let $\mu$ be a partition with
$\ell(\mu)\leq n$ and $\mu_1-\mu_n\leq p-1$.

(1) If $\mu_1\leq k$, then the graded composition multiplicity of
$L((q-1)\rho-\mu+k\om_n)$ in $\sym$ is
    \begin{align*}
    \ds\frac{t^{-n+(k+1)\frac{q^n-1}{q-1}}}
    {\prod^n_{i=1}(1-t^{q^i-1})}~m_{\mu}(t^{-1},t^{-q},\ldots,t^{-q^{n-1}}).
    \end{align*}

(2) If $\mu_1+k<q-1$, then the graded composition multiplicity of
$L((q-1)\rho+w_0\mu+k\om_n)$ in $\sym$ is
\begin{align*}
\ds\frac{t^{-n+(k+1)\frac{q^n-1}{q-1}}}
{\prod^n_{i=1}(1-t^{q^i-1})}~m_{\mu}(t,t^{q},\ldots,t^{q^{n-1}}).
\end{align*}
\end{thm}
\begin{proof}
%

(1) Suppose $\mu_1\leq k$.  We shall prove by induction on dominance
order of $\mu$ the equivalent claim that the Hilbert series of ${\rm
Hom}_{\Gq}(P((q-1)\rho-\mu+k\om_n), \sym)$ is given by
\begin{align*}
    \ds\frac{t^{-n+(k+1)\frac{q^n-1}{q-1}}}
    {\prod^n_{i=1}(1-t^{q^i-1})}~m_{\mu}(t^{-1},t^{-q},\ldots,t^{-q^{n-1}}).
\end{align*}

For $\mu$ minimal in the sense that it has no partition $\tau<\mu$
with $\ell (\tau) \le n$,  by~(\ref{eqn:em}) we have
$$
e_{\mu'}(t^{-1},t^{-q},\ldots,t^{-q^{n-1}})=m_{\mu}(t^{-1},t^{-q},\ldots,t^{-q^{n-1}})
$$
and moreover by Corollary~\ref{cor:decompStIrr*} we have
$$
P((q-1)\rho-\mu+k\om_n)\cong \St\otimes \wedge^{\mu'}(V)^*\otimes\Det^k.
$$
Hence
\begin{align*}
{\rm Hom}_{\Gq} & (P((q-1)\rho-\mu+k\om_n),\sym)\\
\cong&{\rm Hom}_{\Gq}(\St\otimes \wedge^{\mu'}(V)^*\otimes\Det^k,\sym)\\
\cong &{\rm Hom}_{\Gq}(\St,\sym\otimes\wedge^{\mu'}(V)\otimes\Det^{q-1-k}).
\end{align*}
The claim for such a minimal weight $\mu$ follows
by Theorem~\ref{symwedgnu}
since $1\leq q-1-k\leq q-1$ and $\mu_1+(q-1-k)\leq q-1$ by the assumption
$\mu_1\leq k$.

If $\tau$ is a partition satisfying $\tau<\mu$,
 then $\tau_1\leq\mu_1\leq k$ and  $\tau_1-\tau_n\leq\mu_1-\mu_n\leq p-1$.
 Hence by inductive assumption the Hilbert series of
${\rm Hom}_{\Gq}(P((q-1)\rho-\tau+k\om_n),\sym)$
for $\tau<\mu$ is given by
\begin{align*}
    \ds\frac{t^{-n+(k+1)\frac{q^n-1}{q-1}}}
    {\prod^n_{i=1}(1-t^{q^i-1})}~m_{\tau}(t^{-1},t^{-q},\ldots,t^{-q^{n-1}}).
\end{align*}
By Corollary~\ref{cor:decompStIrr*} and Theorem~\ref{symwedgnu}
the Hilbert series of ${\rm Hom}_{\Gq}(P((q-1)\rho-\mu+k\om_n),\sym)$
has the form
\begin{align*}
&\ds\frac{t^{-n+(k+1)\frac{q^n-1}{q-1}}}
    {\prod^n_{i=1}(1-t^{q^i-1})}\Big(~e_{\mu'}(t^{-1},t^{-q},\ldots,t^{-q^{n-1}})
-\sum_{\tau<\mu}
~a_{\mu'\tau}m_{\tau}(t^{-1},t^{-q},\ldots,t^{-q^{n-1}})\Big)\\
=&\ds\frac{t^{-n+(k+1)\frac{q^n-1}{q-1}}}
    {\prod^n_{i=1}(1-t^{q^i-1})}~m_{\mu}(t^{-1},t^{-q},\ldots,t^{-q^{n-1}}).
\end{align*}
Here we have used the identity that
$
e_{\mu'}  =\sum_{ \tau\leq\mu}a_{\mu'\tau}m_{\tau}
$
and $a_{\mu'\mu}=1$.
Therefore the first part of theorem is proved.

The second part of
the theorem follows from a similar argument together using
Proposition~\ref{decompStIrr} and Corollary~\ref{cor:symwedgnu*}.
\end{proof}

\begin{rem}
There is a duality formulated in \cite[Proposition~2.12]{CW} between
the graded composition multiplicity of $L(\mu)$ and that of
$L(\mu)^*$ in general. The two parts of
Theorem~\ref{thm:grcompmult2} fit well with such a duality.
\end{rem}


\subsection{The coinvariant algebra}

According to a classical theorem of Dickson \cite{Di}, the algebra
of $\Gq$-invariants $\sym^{GL_n(q)}$ is a polynomial algebra in $n$
generators, and its Hilbert series is given by
\begin{equation}   \label{eq:Dickson}
\frac{1}{\prod^{n-1}_{i=0}(1-t^{q^n-q^i})}.
\end{equation}

Consider the following quotient algebra
\begin{equation*}
S^\bullet (V)_{GL_n(q)} :=\sym/I^\bullet_+,
\end{equation*}
where $I^\bullet_+$ denotes the ideal of $\sym$ generated by
homogeneous elements of positive degree in $\sym^{GL_n(q)}$. The
graded algebra $S^\bullet(V)_{GL_n(q)}$ is called the {\em
coinvariant algebra} for $\Gq$. We recall the following basic
results of Mitchell.

\begin{lem} \cite[Proposition~1.3, Theorem~1.4]{Mi}   \label{lem:Mitchell}
As $\Gq$-modules,

(1)
$\sym$  has the same composition series as $\sym^{\Gq}\otimes
\sym_{\Gq}$;

(2) $\sym_{\Gq}$ has the same composition series as  the regular
representation $\F\Gq$.
\end{lem}
Mitchell further pointed out that $\sym_{\Gq}$ is not isomorphic to
$\F\Gq$ since $\sym_{\Gq}$  has the trivial module (in degree zero)
as a direct summand.

Thanks to Lemma~\ref{lem:Mitchell}, Theorem~\ref{thm:grcompmult2}
admits the following reformulation in terms of the coinvariant algebra.

\begin{thm}  \label{th:coinv}
Suppose $p>n$ and $0\leq k\leq q-2$.
Let $\mu$ be a partition with $\ell(\mu)\leq n$ and $\mu_1-\mu_n\leq p-1$.

(1) If $\mu_1\leq k$,
then the graded composition multiplicity of
$L((q-1)\rho-\mu+k\om_n)$ in the coinvariant algebra $S^\bullet(V)_{GL_n(q)}$ is
    \begin{align*}
    \ds\frac{t^{-n+(k+1)\frac{q^n-1}{q-1}}\prod^{n-1}_{i=0}(1-t^{q^n-q^i})}
    {\prod^n_{i=1}(1-t^{q^i-1})}~m_{\mu}(t^{-1},t^{-q},\ldots,t^{-q^{n-1}}).
    \end{align*}

(2) If $\mu_1+k<q-1$, then the graded composition multiplicity of
$L((q-1)\rho+w_0\mu+k\om_n)$ in the coinvariant algebra $S^\bullet(V)_{GL_n(q)}$ is
\begin{align*}
\ds\frac{t^{-n+(k+1)\frac{q^n-1}{q-1}}\prod^{n-1}_{i=0}(1-t^{q^n-q^i})}
{\prod^n_{i=1}(1-t^{q^i-1})}~m_{\mu}(t,t^{q},\ldots,t^{q^{n-1}}).
\end{align*}
\end{thm}

Observe that the limit as $t \mapsto 1$ of either formula in the above theorem
is equal to
\begin{align*}
&=\ds
\frac{\prod^{n-1}_{i=0} (q^n-q^i)}{\prod^n_{i=1}(q^i-1)}  m_{\mu} (1,1,\ldots,1) \\
&={\rm dim~St}\cdot m_{\mu} (1,1,\ldots,1)   \\
&={\rm dim}~\St\cdot|W^{\mu}|,
\end{align*}
which is the dimension of the corresponding projective cover by
Lemma~\ref{lem:Ballard}. This is consistent with
Lemma~\ref{lem:Mitchell}, since the composition multiplicity of a
simple module in the regular representation of a finite group is always equal
to the dimension of its projective cover.

\subsection{Some open problems}

Theorem~\ref{thm:grcompmult1} and Theorem~\ref{thm:grcompmult2} have
provided partial answers to the problem of finding the graded
composition multiplicity of an irreducible $\Gq$-module in the
symmetric algebra $\sym$. They are obtained by converting the
computations of the Steinberg module multiplicity in $\sym \otimes
\wedge^\nu(V)$ in Section~\ref{sec:Stmult} and the results of Ballard
and Tsushima.

\begin{question}
How to decompose $\St \otimes N$ into a direct sum of PIMs for a
reasonable $\Gq$-module $N$? Can we relax the restriction on $p$?
\end{question}
Suitable generalizations of results  of Ballard and Tsushima in
answer to the above question would allow one to expand the range of
applicability of the approach developed in this paper. The methods
developed in this paper seem likely to apply to the following.

\begin{question}
Find the composition multiplicity of the Steinberg module (or the
simple modules around it) in the symmetric algebra of the natural
module for other classical finite groups of Lie type.
\end{question}

The $\Gq$-module $\sym$ is not semisimple, and it makes sense to ask
the following.
\begin{question}
What is the graded multiplicity of a simple $\Gq$-module $L(\mu)$ in
the socle of $\sym$?
\end{question}

Dickson's classical theorem \cite{Di} (see \eqref{eq:Dickson})
provides a first beautiful answer in case when $L(\mu)$ is the
trivial module. Several generalizations have been obtained in
\cite{Mui, Mi, MP, KM}, culminating in our previous work \cite{WW}
which settled this socle mutiplicity question for the simple modules
of the form $\wedge^m(V) \otimes \Det^k$ for arbitrary prime powers
$q=p^r$, $1\le m \le n$ and $0 \le k \le q-2$. The answer in {\em
loc. cit.} fit into the following form, which we ask if it holds for
a wider class of $L(\mu)$:

Let $\mu=(\mu_1,\ldots,\mu_n)$ be a partition of $d$ with $1\leq
d\leq p-1$ and $0\leq k\leq q-2-\mu_1$. Is the multiplicity of
the simple module $L(\mu)\otimes \Det^k$ in the socle of $\sym$
given by
\begin{align*}
\frac{t^{k\cdot \frac{q^n-1}{q-1}}}
{\prod^{n-1}_{i=0}(1-t^{q^n-q^i})} \cdot
s_{\mu}(t,t^q,\ldots,t^{q^{n-1}})?
\end{align*}

%

\end{document}